\documentclass[11pt]{amsart}
\usepackage{xcolor}

\newtheorem{theorem}{Theorem}[section]
\newtheorem{lemma}{Lemma}[section]
\newtheorem{proposition}{Proposition}[section]

\theoremstyle{definition}
\newtheorem{definition}{Definition}[section]
\theoremstyle{remark}
\newtheorem{remark}{Remark}[section]
\numberwithin{equation}{section}
\newtheorem{corollary}{Corollary}[section]

\newtheorem{problem}{Problem}[section]

\numberwithin{equation}{section}

\newcommand{\R}{\mathbb R}

\begin{document}

 \title[Critical points]{Maps with finitely many critical points into high dimensional manifolds}
\author[L.Funar]{Louis Funar}
\address{Institut Fourier BP 74, UMR 5582, Laboratoire de Math\'ematiques, 
Universit\'e Grenoble Alpes, CS 40700, 38058 Grenoble cedex 9, France}
\email{louis.funar@univ-grenoble-alpes.fr}

\begin{abstract} Assume that there exists a smooth map between two closed manifolds 
$M^m\to N^k$, where $2\leq k\leq m\leq 2k-1$, with only finitely many singular points, all of which are 
cone-like. 
If $(m,k)\not\in\{(2,2), (4,3), (5,3), (8,5), (16,9)\}$, then $M^m$ admits a locally trivial topological fibration over $N^k$ and there exists a smooth map  $M^m\to N^k$ with at most one critical point. \\

\noindent Keywords: isolated singularity, open book decomposition, Poenaru-Mazur contractible manifold.

\noindent MSC Class: 57R45, 57 R70, 58K05.
\end{abstract}

\maketitle

\section{Introduction and statements}

Let $\varphi(M^m,N^k)$ denote the minimal number of critical points
of a smooth map between the manifolds $M^m$ and $N^k$ of dimensions 
$m\geq k\geq 2$. Church and Timourian asked (see \cite{CT1}, p.617) for a characterization of those
pairs of manifolds for which $\varphi$ is finite non-zero, called {\em nontrivial} here, 
and then to compute its value.

In \cite{AndFun1} the authors found that,
in {\em very small codimension}, namely when $0\leq m-k\leq 3$, if $\varphi(M^m,N^{k})$ is finite, 
then $\varphi(M^m,N^{k})\in\{0,1\}$, except for the
exceptional pairs of dimensions
$(m,k)\in\{(2,2), (4,2), (4,3), (5,2), (6,3), (8,5)\}$.
Moreover, $\varphi(M,N)=1$ if and only if $M$ is diffeomorphic to the connected 
sum of a  smooth locally trivial fibration over $N$ with an exotic sphere while $M$ 
is not a smooth fibration over $N$.

The value of $\varphi$ was computed for (pairs of) surfaces 
in \cite{AFK}. 
The dimensions $(4,3), (8,5)$ and $(16,9)$ were analyzed in \cite{FPZ} where it was 
shown that $\varphi$ can take arbitrarily large (even) values. 
Moreover, all non-trivial isolated singularities  in 
dimensions $(4,3)$ and $(8,5)$ are of a very particular nature, 
namely they are suspension of Hopf fibrations; this permitted to classify the pairs of manifolds 
with finite $\varphi$ in these dimensions (see \cite{F}). 

Nontrivial examples were constructed in \cite{FP}, in all dimensions $(m,k)$ with $m\geq 2k$, also called  the {\em high codimension}. 
Lower bounds for $\varphi(M^m,N^k)$, when $m=2k-2$, have been obtained in \cite{FPZ}  
for many examples, but upper bounds could only be provided for $k\in\{3,5,9\}$.

In \cite{AFK,FP} a relevant invariant  $\varphi_c$, which counts critical points of maps having only 
finitely many singularities which are {\em cone-like}, was introduced. 
In very small codimension $\varphi_c$ agrees with $\varphi$.
Conjecturally this should be true in any dimension (see the last section).

The aim of this paper is to consider the  
{\em small codimension} case, namely  $m\leq 2k-1$ and $k\geq 3$ by showing that 
all pairs with finite $\varphi_c(M,N)$  are topologically trivial, except for the exceptional dimensions. 
The idea of the proof is as follows. We recall that Church and Lamotke (\cite{CL}) provided a general  
construction of non-trivial local isolated singularities in all codimensions  $m-k\geq 4$. 
Their examples  answered in the negative 
a conjecture of Milnor  (\cite{Milnor}, p. 100) claiming that all  isolated singularities 
with $m\leq 2k-3$ are trivial in the sense that their local fibers are homeomorphic to disks, so they are locally topologically equivalent to regular maps. 
We first show that every cone-like isolated singularity 
actually arises by means of their construction, using Haefliger's unknotting Theorem. 
On the other hand we prove that 
cone-like isolated singularities can be removed by disk surgeries. The topology 
of the source manifold changes by connected sums with homotopy spheres.  
Therefore, critical points could contribute by at most one unit to $\varphi_c$. 
This shows that a suitable global version of Milnor's conjecture above holds, namely 
maps $M^m\to N^n$ with finitely many critical points only exist when topological submersions 
$M^m\to N^n$ exist (in these dimensions).


\begin{definition}
A  singularity $p$ of the smooth map $f:M\to N$ is {\em cone-like} 
if $p$ admits a cone neighborhood in the corresponding fiber $V=f^{-1}(f(p))$, i.e. there exists some closed manifold 
$L\subset V-\{p\}$ and a neighborhood $U$ of $p$ in $V$ which is homeomorphic to 
the cone $C(L)$ over $L$. Recall that $C(L)=L\times (0,1]/L\times \{1\}$. 
\end{definition}

Church and Timourian proved  (see \cite{CT2} and another proof in \cite{F}) 
that all isolated singularities in codimensions at most $2$ are cone-like. 
As Takens pointed out in \cite{Takens}, this is not anymore true in codimension 3 or higher. 
However, isolated singularities in codimension 3 (and dimension of the base at least $4$) are 
removable according to \cite{F}.

\begin{theorem}\label{smallcodimension}
Suppose we are in the small codimension case $k\geq 2$ and $m\leq 2k-1$. 
We suppose that  $(m,k)\not\in\{(2,2), (4,3), (8,5), (16,9)\}$. 
Assume that there exists a smooth map between two closed manifolds 
$f:M^m\to N^k$ with only finitely many singular points, all of which are 
cone-like. 
If $(m,k)=(5,3)$ we suppose that the critical points are regular (see section \ref{fiberedlinks} for definition). 
Then $M^m$ admits a locally trivial topological fibration over $N^k$ and 
 $\varphi_c(M^m,N^k)\in\{0,1\}$. 
Moreover, $\varphi_c(M^m,N^k)=1$ if and only if $M^m$ is the connected sum of a 
smooth fiber bundle over $N^k$ with an exotic $m$-sphere, while $M^m$ does not fiber smoothly  
over $N^k$.  
\end{theorem}

Recall that the behavior of $\varphi$ in the excluded  
cases $m=2k-2$ and $k\in\{2,3,5,9\}$ is different. In fact, the main result of \cite{FPZ} states that: 
\[ \varphi(\sharp_{e}S^{k-1}\times S^{k-1}\sharp_{c}S^1\times S^{2k-3}, \; \sharp_{c}S^1\times S^{k-1})=
2e-2c+2, \;  {\rm for} \; k\in \{3,5,9\}\]
Here $\sharp_mM^{2k-2}$ denotes the connected sum of $m$ copies of the manifold $M^{2k-2}$, when 
$m\geq 1$, and $S^{2k-2}$, when $m=0$, respectively.

In contrast, we have the following consequence of (\cite{FPZ}, Prop. 2.1) and Theorem 
\ref{smallcodimension}:  
\begin{corollary}
Suppose that $m=2k-2$ and $k\not\in\{2,3,5,9\}$.
Then for $e\geq c\geq 0$, $c\neq 1$ we have: 
\[ \varphi(\sharp_{e}S^{k-1}\times S^{k-1}\sharp_{c}S^1\times S^{2k-3}, \; \sharp_{c}S^1\times S^{k-1})=
\infty\]
\end{corollary}

\section{Proof of Theorem \ref{smallcodimension}}

\subsection{Fibered links and local models for isolated singularities}\label{fiberedlinks}
We use the terminology of Looijenga (\cite{Loo}) and say that the isotopy class of the oriented 
submanifold $K=K^{m-k-1}$ of dimension $(m-k-1)$ of $X^{m-1}$ with a trivial 
normal bundle  is {\em generalized Neuwirth-Stallings fibered} (alternatively $(X^{m-1}, K^{m-k-1})$ 
is a  generalized Neuwirth-Stallings pair)   
if, for some trivialization $\theta: N(K)\to K\times D^k$ of the tubular 
neighborhood $N(K)$ of $K$ in $X^{m-1}$,  
the fiber bundle $\pi\circ \theta: N(K)-K \to S^{k-1}$ admits an extension 
to  a smooth fiber bundle 
$f_K:X^{m-1}- K\to S^{k-1}$.  We denoted above by $\pi:K\times (D^k-\{0\})\to S^{k-1}$ the 
composition of the radial projection $r: D^k-\{0\}\to S^{k-1}$ 
with the second factor projection.   
The data $\mathcal E=(X^{m-1}, K, f_K, \theta)$ is then called an {\em open book decomposition} with binding $K$, while $K$ is called a {\em fibered link}.
Note that a fibration of $X^{m-1}- K\to S^{k-1}$ comes from an open book decomposition 
if and only if the closure of every fiber is its compactification by the binding link.
We will still denote by $f_K$  the induced fibration $f_K:X^{m-1}- N(K)\to S^{k-1}$, whose fiber is now compact. The classical notions of Neuwirth-Stallings fibrations and pairs correspond to 
$X^{m-1}=S^{m-1}$.

An {\em adapted neighborhood} around a cone-like isolated  critical point of a 
map $f:M^m\to N^k$ is a compact manifold neighborhood $Z^{m}\subset M^m$ containing it 
with the following properties:

\begin{enumerate}
\item $f$ induces a proper map $f:Z^m\to D^k$ onto a $k$-ball $D^k\subset N^k$; 
\item $f^{-1}(x)$ is transverse to $\partial Z^m$ for $x\in{\rm int}(D^k)$ and 
$E=f^{-1}(S^{k-1})\cap Z^m\subset \partial Z^m$;   
\item Let 
$K^{m-k-1}=\partial Z^m\cap f^{-1}(0)$, $D_0^k\subset D^k$ be a small disk around $0$ and $N(K)=\partial Z^m \cap f^{-1}(D_0^k)$ 
be a tubular neighborhood of $K$ within $\partial Z^m$, 
endowed with the trivialization $\theta$ induced by $f$; 
\item the composition $f_K=r\circ f: \partial Z^m- f^{-1}(0)\to S^{k-1}$ 
of the radial projection $r$ with $f$ is a locally trivial fiber bundle;
\item the data $(\partial Z^m, K, f|_{K}, \theta)$ is an open book decomposition.    
\end{enumerate}

Critical points are called {\em regular} if there exist adapted 
neighborhoods diffeomorphic to $D^m$.

We summarize King's results from \cite{King,King2} as follows:
\begin{lemma}\label{adapted}
Cone-like isolated singularities of smooth maps $f:\R^m\to \R^n$ admit 
adapted neighborhoods. Moreover,  when $m\neq 4,5$  
cone-like isolated singularities are regular. 
\end{lemma}

Recall now from \cite{KN,Loo}  that an open book decomposition $\mathcal E=(S^{m-1}, K, f_K, \theta)$ gives rise to a local  
isolated singularity $\psi_{\mathcal E}:(D^m,0)\to (D^k,0)$  by means of the formula:
\[\psi_{\mathcal E}(x)=\left\{\begin{array}{cl}
\lambda(\|x\|)f_K\left(\frac{x}{||x||}\right), & {\rm if} \; \frac{x}{||x||}\not\in N(K); \\
\lambda\left( \|x\|\cdot \left\|\pi_2\left(\theta\left(\frac{x}{\|x\|}\right)\right)\right\|\right)f_K\left(\frac{x}{||x||}\right), & {\rm if} \; \frac{x}{||x||}\in N(K); \\
0, & {\rm if }\;  x=0,
\end{array}
\right.
\]
where $\pi_2:K\times D^k\to D^k$ is the projection on the second factor and $\lambda:[0,1]\to [0,1]$ is any smooth 
strictly increasing map sufficiently flat at $0$ and 1 such that $\lambda(0)=0$ and 
$\lambda(1)=1$. Although $\psi_{\mathcal E}$ is not uniquely defined by this formula, all 
germs obtained this way are topologicaly equivalent. This is a direct consequence of the characterization of cone-like isolated singularities due to King 
(see \cite{King}, Thm. 2 and \cite{King2}, Thm. 1).  Moreover, they are also smoothly equivalent 
by germs of diffeomorphisms of $D^m\setminus\{0\}$ (see \cite{KN}, Thm. 1.10, for $k=2$).  
We call such $\psi_{\mathcal E}$ {\em local models} of isolated singularities.

If  $K$ is in generic position, namely the space generated by vectors in $\R^m$ with endpoints in $K$ coincides with 
the whole space $\R^m$, then  $(d\psi_{\mathcal E})_{0}=0$, i.e. $\psi_{\mathcal E}$ has rank $0$ at the origin.

By language abuse we will speak of fibered links $K\subset S^{m-1}$ as being links which 
admit an open book structure. We emphasize that $\psi_{\mathcal E}$ depends on the 
choice of the fibration $f_K$ and the trivialization $\theta$, not of the 
isotopy class of the embedding of $K$ in $S^{m-1}$ alone. 
However, in the case $k=2$, when $m=2n\geq 8$ is even and 
$K^{2n-3}\subset S^{2n-1}$ is a $(n-3)$-connected  fibered knot, Durfee and Kato proved that 
any two fibrations of $S^{2n-1}-K$ are bundle equivalent (see \cite{Dur}, Cor. 3.3 and \cite{Kato}). 

Looijenga proved in \cite{Loo} that a Neuwirth-Stallings pair $(S^{m-1},L^{m-k-1}, f_L, \theta)$  can be realized by a {\em real polynomial} map  
if $L$ is invariant and the open book fibration $f_L$ is equivariant with respect to the antipodal maps.  
In particular, for any fibered link $K$ the connected sum 
$(S^{m-1},K)\sharp ((-1)^{m}S^{m-1}, (-1)^{m-k}K)$ is a Neuwirth-Stallings pair isomorphic to the 
link of a real polynomial isolated singularity $\psi_K:(\R^m,0)\to (\R^k,0)$.

\subsection{Fibered links in codimension at least $4$, after Church and Lamotke}\label{half}
In \cite{CL} Church and Lamotke constructed fibered links $(S^{m-1}, K^{m-k-1})$  in any dimensions
$m,k$ with $m-k\geq 4$ and $k\geq 2$. Note that the existence of fibered links for  $k=1$ is well-known.

A compact contractible smooth $n$-manifold $F^{n}$ with non-trivial $\pi_1(\partial F)$ is called a 
{\em Poenaru-Mazur} manifold. Examples of Poenaru-Mazur manifolds were first constructed  
in dimension $4$ (see \cite{Mazur,Poe}) and further extended to all dimensions $n\geq 5$ 
by Curtis (see \cite{Curtis}).  Since $F^{n}$ are contractible they have smooth structures. 

Consider a Poenaru-Mazur manifold $F^{m-k}$, $m-k\geq 4$. Note that 
$F^{m-k}\times [0,1]$ is homeomorphic to the $(m-k+1)$-disk $D^{m-k+1}$. Indeed  
$F^{m-k}\times [0,1]$  is a combinatorial manifold whose boundary is homotopy equivalent 
to a $S^{m-k}$.  By the Stallings-Zeeman solution of the Poincar\'e Conjecture 
for combinatorial manifolds (see \cite{St,Zee}),  this boundary is  homeomorphic to $S^{m-k}$. 
Furthermore any compact contractible combinatorial manifold with boundary $S^{m-k}$ is 
PL homeomorphic to $D^{m-k+1}$ (see \cite{Curtis}).

Note that $F^{m-k}\times D^k$, for $k\geq 2$ is a smooth manifold with corners, 
whose corners can be straightened. As $F^{m-k}\times D^k$  is homeomorphic to $D^m$ 
it is a smooth contractible manifold with simply connected boundary.  
By a classical result of Smale  (see \cite{Smale}, Thm. 5.1)  
$F^{m-k}\times D^k$  is  diffeomorphic to $D^m$, as soon as $m\geq 6$.
 
Let $\phi: F^{m-k}\times D^k\to D^m$ be such a diffeomorphism and set 
$K^{m-k-1}=\phi(\partial F^{m-k}\times \{0\})\subset \phi(\partial (F^{m-k}\times D^k))=S^{m-1}$. 
Let $D_0^k\subset D^k$ be a small disk around $0$.
Then we consider a tubular neighborhood of $K$ within $S^{m-1}$ of the form 
$N(K)=\phi(\partial F^{m-k}\times D_0^k)$. It has a 
trivial trivialization $\theta:N(K)\to K\times D^{k}$ given by 
$\theta(\phi(x))=\pi_2(x)$, where $\pi_2$ denotes the second factor projection. 

Now we have the  identifications: 
\[S^{m-1}-K = \phi(\partial (F^{m-k}\times D^k) -\partial F^{m-k}\times \{0\}) =   \phi(\partial F^{m-k}\times (D^k-\{0\})) \cup \phi(F^{m-k} \times \partial D^{k})\]
Let $\pi_K:\partial F^{m-k}\times (D^k-\{0\})\to S^{k-1}$ be the composition of the radial
projection $D^k-\{0\}\to S^{k-1}$ with the second factor projection and 
$\pi_2: F^{m-k}\times \partial D^{k} \to S^{k-1}$ be the second factor projection. 
We define then $f_K: S^{m-1}-K\to S^{k-1}$  by 
\[f_K (\phi(x)) = \left\{\begin{array}{ll}
\pi_K(x), & {\rm if } \; x \in \partial F^{m-k}\times (D^k-\{0\}); \\
\pi_2(x), & {\rm if } \; x \in F^{m-k}\times \partial D^k.\\
\end{array} 
\right.
\]
It follows that $\mathcal E_{F}=(S^{m-1}, K^{m-k-1}, f_K, \theta)$ is an open book decomposition with binding $K^{m-k-1}$. Although implicit, the choice of the diffeomorphism $\phi$ enters 
in the definition of $\mathcal E_{F}$. 

The previous construction yields a smooth map  $\psi_{\mathcal E_{F}}:(D^m,0)\to (D^k,0)$ with an isolated singularity at the origin.

Note that the pair $(\psi_K^{-1}(0)\cap D^m, \psi_K^{-1}(0)\cap S^{m-1})$  is homeomorphic to 
$(C(K),K)$, where $C(K)$ denotes the cone over $K$. 
The  Cannon-Edwards Theorem (see \cite{W}) states that a polyhedral homology manifold is homeomorphic to a topological manifold  
if and only if the link of every vertex is simply connected. Now, $C(K)$ is a polyhedral homology manifold and 
a PL manifold outside $0$; since the link at $0$ is homeomorphic to $K$ and $\pi_1(K)\neq 0$, we derive that 
the the singular fiber is not homeomorphic to a topological manifold. In particular, germs of isolated singularities 
arising by this construction are not topologically equivalent to trivial germs associated to non-singular points.

\subsection{Cut and paste  local models}\label{cutpaste}
The method used in \cite{F,FPZ,FP} to globalize a local picture was to glue together a patchwork of such local models 
to obtain maps $M^{m}\to N^{k}$ with finitely many 
critical points.   

Let  $(S^{m-1}, K_j^{m-k-1})$ be  fibered knots with fibers the contractible manifolds $F_j^{m-k}$. Consider some   $(m-k)$-manifold 
$Z^{m-k}$ with $\partial Z^{m-k}=\sqcup_j K^{m-k-1}$. 
We glue together  copies $D_j^m$  of the $m$-disk with the product  $Z\times D^k$ along part of their boundaries 
by identifying  $\sqcup_j N(K_j)$ with $\partial Z\times D^k$. 
The identification has to respect the trivializations  
$\sqcup_j N(K_j)\to D^k$  and hence one can take them to be 
the same as in the double construction. Note that $N(K_j)=K_j\times D^{k}$ and thus identifications respecting the trivialization 
correspond to homotopy classes $[\sqcup_j K_j, {\rm Diff}(D^k,\partial)]$.

We then obtain  a manifold with boundary $X=(\cup_j D_j^m)\cup (Z\times D^k)$
endowed with a smooth map $f: X\to D^k$  with finitely many critical points lying in the same  singular fiber.  
Its generic fiber is the manifold $F$ obtained by 
gluing together  $(\cup_jF_j)\cup Z$. 

Let now $f_1,f_2,\ldots, f_p$ be a set of such maps which   
are {\em cobounding}, namely such that there exists a fibration over 
$N^k\setminus \sqcup_{i=1}^{p}D^k$, generally not unique, extending the boundary 
fibrations restrictions $f_i|_{\partial X_i}$, $1\leq i\leq p$. 
Gluing together the boundary fibrations we obtain    
a closed manifold $M(f_1,f_2,\ldots,f_p)$
endowed with a map with finitely many critical points into $N^k$. 
In particular, we can realize the double 
of ${f}$ by gluing together ${f}$ and its mirror image.

\subsection{Topological submersions in small codimension}

\begin{lemma}\label{contractible}
In small codimension $m-k\leq k-1$, $k\geq 2$ if $f:D^m\to D^k$ has a single cone-like 
critical point and $D^m$ is an adapted neighborhood of it, 
then either the generic fiber $F^{m-k}$ is contractible or else the link $K$ is empty.  
\end{lemma}
\begin{proof}
Suppose that the link is non-empty. The hypothesis amounts to say that there 
exists an open book decomposition of $S^{m-1}$  with fiber $F^{m-k}$ and  binding 
$K^{m-k-1}=\partial F^{m-k}$.  Now the proof follows directly from the arguments used in 
 (\cite{Milnor}, Lemma 11.4). We first have: 
 \[ \pi_j(S^{m-1}-N(K))=\pi_j(S^{m-1})=0, \; {\rm for}\; j\leq k-2\]
 because we obtain $S^{m-1}$  from $S^{m-1}-N(K)$ by adjoining one $(k+j)$-cell for each 
 $j$-cell of $K$, so that lower dimensions homotopy groups agree. 
 The fibration  $f_K: S^{m-1}-N(K)\to S^{k-1}$  admits a cross-section. 
As it is well-known, the long exact sequence  in homotopy associated to a fibration 
breaks down for fibrations with cross-sections 
to a family of short exact sequences which are split exact: 
 \[ 0 \to \pi_j(F^{m-k}) \to \pi_j(S^{m-1}-N(K)) \to \pi_j(S^{k-1})\to 0\]
Therefore $F^{m-k}$ is $(k-2)$-connected. When $m-k\leq k-2$, it follows that 
the manifold  $F^{m-k}$ must be contractible. If $m-k=k-1\geq 2$, then by the Hurewicz Theorem  
the natural morphism $\pi_{k-1}(F^{m-k})\to H_{k-1}(F^{m-k})$ is an isomorphism. As $F^{m-k}$ 
is a $(k-1)$-manifold with boundary  $H_{k-1}(F)=0$ and hence $F^{m-k}$ is $(k-1)$-connected and hence contractible. This also holds when $m-k=k-1=1$, since 
$F$ is a 1-dimensional manifold with boundary. 
 \end{proof}

\begin{remark}
The result of Lemma \ref{contractible} is sharp: as soon as $m-k\geq k$ there exist fibered links 
$K^{m-k-1}\subset S^{m-1}$ whose associated fibers $F^{m-k}$ are homotopy equivalent to a 
(non-trivial) bouquets of spheres (see e.g. \cite{AHSS,FP}). 
\end{remark}

\begin{lemma}\label{empty}
The link $K$ could be empty only if $m=2k-2$ and $k\in\{2,3,5,9\}$.
\end{lemma}
\begin{proof}
In this situation there exists a fibration $S^{m-1}\to S^{k-1}$ and fibrations between spheres are well understood (see e.g. \cite{AndFun1}, Prop. 6.1).
\end{proof}

A key ingredient is the following unknotting result due to Haefliger:  

\begin{lemma}\label{unknotting}
Let $K^n$ be an integral homology $n$-sphere and $2m\geq 3n+4$. 
Then two smooth embeddings of $K^n$ into $S^m$ are smoothly isotopic. 
\end{lemma}
\begin{proof}
This is a particular case of the Haefliger-Zeeman unknotting Theorem which 
is stated in (\cite{Hae}, p. 66): for every $n\geq 2p+2$, $m\geq 2n-p+1$ the 
smooth embeddings of a closed homologically $p$-connected $n$-manifold into $S^m$ 
are smoothly isotopic.   
\end{proof}

\begin{lemma}\label{diffproduct}
Let $\mathcal E=(S^{m-1}, K^{m-k-1}, f_K, \theta)$ be an open book decomposition 
in small codimension $m-k\leq k-1$,  $m\geq 6$. 
Then $S^{m-1}-N(K^{m-k-1})$ is diffeomorphic to the product $F^{m-k}\times S^{k-1}$, where 
$F^{m-k}$ denotes the fiber.  
\end{lemma}
\begin{proof}
Recall that in small codimension case $K^{m-k-1}$ is an integral homology sphere since 
it bounds a contractible manifold. 
By Lemma \ref{unknotting} every two smooth embeddings of 
$K^{m-k-1}$ into $S^{m-1}$ are smoothly isotopic in the metastable range
$2(m-1) \geq 3(m-k-1)+4$, which is our case.     

On the other hand there is a diffeomorphism between 
$F^{m-k}\times D^k$ and $D^m$, if $m\geq 6$ which induces an embedding of 
$\partial F^{m-k}\times D^k$ into $S^{m-1}$. Its complement in this embedding is 
$F^{m-k}\times S^{k-1}$. By above, $S^{m-1}-N(K^{m-k-1})$ is diffeomorphic to $F^{m-k}\times S^{k-1}$
for any embedding of $K^{m-k-1}$ into $S^{m-1}$.   
\end{proof}

\begin{lemma}\label{trivialfibration}
Let $k\geq 3$ and $F^{m-k}$ a compact simply connected manifold. 
Then every smooth locally trivial fibration 
of $F^{m-k}\times S^{k-1}$ over $S^{k-1}$ is smoothly equivalent to a trivial 
fibration. 
\end{lemma}
\begin{proof}
Let $f:F^{m-k}\times S^{k-1}\to S^{k-1}$ be a submersion. 
Define $\varphi: F^{m-k}\times S^{k-1}\to F^{m-k}\times S^{k-1}$, by 
$\varphi(x,y)=(x,f(x,y))$.   
Then $\varphi$ is a local diffeomorphism because 
its differential $D\varphi$ is an isomorphism at each point. 
As its domain and range are compact and connected the  
map $\varphi$ is proper and hence a finite smooth covering.   
Since $\pi_1(F\times S^{k-1})=0$, every  smooth covering map 
is a diffeomorphism. Observe now that $\varphi$ 
provides an equivalence between the fibration given by $f$ 
and the product fiber bundle.     
\end{proof}

\begin{lemma}\label{topgluing}
Let $f: M^m\to N^k$ be a function 
with only finitely many singular points, all of which are 
cone-like. 
Assume that $m\leq 2k-1$ and $(m,k)\not\in\{(2,2), (4,3), (8,5), (16,9)\}$. 
If $(m,k)=(5,3)$ we suppose that the critical points are regular. 
Then $M^m$ is homeomorphic to a patchwork $M(f_1,f_2,\ldots,f_p)$ of local models as above. 
\end{lemma}
\begin{proof}
By King's Lemma \ref{adapted}  there exist adapted regular 
neighborhoods diffeomorphic to $m$-disks. 
The restriction $f:D^m\to D^k$  to an adapted neighborhood $D^m$ of some critical point of $f$ 
provides a fibered link $K=f^{-1}(0)\cap S^{m-1}$, whose  open book 
decomposition $\mathcal E$ is as follows. We have $N(K)=f^{-1}(D_0^k)\cap S^{m-1}$
with the trivialization $\theta: N(K)\to K\times D_0^k$ given by 
the restriction of $f|_{N(K)}$ and a fibration $f:S^{m-1}-N(K)\to S^{k-1}$. 
Moreover, $f$ is smoothly equivalent in $D^m$ to the map $\psi_{\mathcal E}$.   
     
Lemma \ref{contractible} and Lemma \ref{empty} show that the fiber $F^{m-k}$ must be contractible. 
Then Lemma  \ref{trivialfibration}  implies that  $f:S^{m-1}-N(K)\to S^{k-1}$ is topologically equivalent to 
the product bundle $\pi_2:F^{m-k}\times S^{k-1}\to S^{k-1}$. 
This implies that $\mathcal E$ is topologically equivalent to the open book decomposition 
$\mathcal E_{F}$ associated to the contractible manifold $F$ by the Church-Lamotke 
construction above. Thus $F^{m-k}$ should be a Poenaru-Mazur manifold and $f$ is locally 
topologically equivalent to $\psi_{{\mathcal E}_F}$.

Observe that the fibered links $K^{m-k-1}$ arising above are integral homology spheres as they bound contractible manifold and thus they are connected. 
Outside finitely many disjoint adapted neighborhoods of critical points 
the map $f$ is a submersion. Then the lemma follows.  
\end{proof}

\begin{proposition}\label{topsubmersion}
Let $f: M^m\to N^k$ be a function with only finitely many singular points, all of which are 
cone-like. 
Assume that $m\leq 2k-1$ and $(m,k)\not\in\{(2,2), (4,3), (8,5), (16,9)\}$. 
If $(m,k)=(5,3)$ we suppose that the critical points are regular. 
Then $M^m$ admits a locally trivial topological fiber bundle over $N^k$. Specifically, 
there exists a topological submersion  $M^m\to N^k$ whose fiber 
is homeomorphic to the generic fiber of $f$. 
\end{proposition}
\begin{proof}
We will show that  we can get rid of the singularities of the map $f$ by changing 
it within the adapted neighborhoods. 

Lemma \ref{topgluing} shows that $M$ is homeomorphic to a patchwork of local models. 
Let $q$ be an isolated cone-like singularity as above  and 
$K$ be the associated fibered link. 
Lemma  \ref{trivialfibration} provides 
a homeomorphism of fibrations $\varphi_K: S^{m-1}-N(K)\to F\times S^{k-1}$. 
Since the boundary of the fibration 
$f_K:S^{m-1}-N(K)\to S^{k-1}$ is trivial and agrees with the trivialization $\theta$ 
we can choose the homeomorphism  $\varphi_K$ to agree with the boundary trivialization 
$\theta|_{\partial N(K)}:\partial N(K) \to \partial F \times S^{k-1}$. Therefore 
the map $\varphi: S^{m-1}\to \partial (F\times D^k)$ given by 
\[\varphi (x) = \left\{\begin{array}{ll}
\varphi_K(x), & {\rm if } \; x \in S^{m-1}-N(K); \\
\theta(x), & {\rm if } \; x \in N(K)\\
\end{array} 
\right.
\]
is a homeomorphism. 

Consider further the manifold  $M_q$ obtained by removing the interior of the 
adapted neighborhood disk $D^m$ around $q$ and then gluing back $D^m$ after twisting  
by $\varphi$, namely: 
\[ M_q = (M- {\rm int}( D^m)) \cup F^{m-k}\times D^k/ {x\sim \varphi(x)}
\]
where the equivalence relation $\sim$ identifies points of $x\in S^{m-1}$ 
to their images $\varphi(x)\in F^{m-k}\times D^k$. 

Now define the map $f_q:M_q\to N^k$ 
\[ f_q(x) = \left\{\begin{array}{ll}
f(x), & {\rm if } \; x \in M- {\rm int}( D^m) ; \\
\pi_2(x) & {\rm if } \; x \in F^{m-k}\times D^k.\\
\end{array} 
\right.
\]
where $\pi_2: F^{m-k}\times D^k\to D^k$ is the second factor projection.

It is immediate that $f_q$ is continuous everywhere and $f_q$ has no more topological 
critical points in the disk $F^{m-k}\times D^k$. Using the disk surgery above 
for the whole set $Q$ of critical points we obtain a map    
$f_Q: M_Q\to N^k$ without critical points and hence a topological submersion. 
Note eventually that the manifold $M_Q$ is homeomorphic to $M$, as a 
disk surgery corresponds to 
making the connected sum with a homotopy sphere. 
\end{proof}

\subsection{Proof of Theorem \ref{smallcodimension}}
We have to show that all steps in the proof above can be carried in the smooth category. 
More precisely, manifolds as in Lemma \ref{topgluing} should be diffeomorphic to patchworks of 
local models and the surgery of disks in the proof of Proposition \ref{topsubmersion} 
should be realized in the smooth category.

\begin{proposition}\label{smoothsubmersion}
Let $f: M^m\to N^k$ be a function with only finitely many singular points, all of which are 
cone-like. 
Assume that $m\leq 2k-1$ and $(m,k)\not\in\{(2,2), (4,3), (8,5), (16,9)\}$. 
If $(m,k)=(5,3)$ we suppose that the critical points are regular. 
Then there exists some homotopy sphere $\Sigma^m$ such that the connected sum 
$M^m\sharp \Sigma^m$ is a  fiber bundle over $N^k$, whose fiber 
is diffeomorphic to the generic fiber of $f$. 
\end{proposition}
\begin{proof}
By Lemmas \ref{adapted}, \ref{diffproduct} and \ref{trivialfibration}   we derive that 
the conclusion of Lemma \ref{topgluing} can be improved to the fact that  
$M^m$ is diffeomorphic to a patchwork $M(f_1,f_2,\ldots,f_p)$ of local models. 
Indeed any open book decomposition in these dimensions is isotopic to the corresponding 
open book decomposition with the same fiber constructed by Church and Lamotke in section \ref{half}.  
In the proof of Proposition \ref{topsubmersion} every disk surgery can now be 
realized using a diffeomorphism as gluing map. This operation corresponds to a connected sum with a
homotopy sphere, as claimed.   
\end{proof}

\vspace{0.3cm}
\begin{proof}[End of proof of Theorem \ref{smallcodimension}]
From Proposition \ref{smoothsubmersion} $M^m$ is diffeomorphic to $E\sharp \Sigma^m$, 
where $g:E^m\to N^k$ is a fiber bundle. Observe that every diffeomorphism 
$S^{m-1}\to S^{m-1}$ extends to a smooth homeomorphism $D^m\to D^m$ with a single critical point
obtained by smoothing the cone map at the origin. This provides a 
smooth map  $\phi:E^m\sharp \Sigma^m\to E^m$  which is a homeomorphism and has a single critical point. Then the composition $g\circ\phi$ is a map with a single critical point. 
When $M^m$ is not a fiber bundle over $N^k$ we must have $\varphi(M^m,N^k)=1$. 
\end{proof}

\section{Comments}
By the Grauert-Morrey Theorem and Whitney approximation Theorem  
smooth manifolds admit unique real-analytic structures; moreover, 
two smooth manifolds which are diffeomorphic are also real-analytic diffeomorphic, with respect to their unique real-analytic structures. 
These results only hold for {\em second countable}  manifolds, as examples of Kneser and 
Kneser (see \cite{KK})  show that the Alexandrov 
line has several  analytically inequivalent real analytic structures.

Looijenga (\cite{Loo})  proved that to any  open book decomposition   
we can associate a semi-algebraic local isolated singularity, 
whose components are of the form $P(x)+|x| Q(x)$, where 
$P, Q$ are real polynomials and $| \; |$ denotes the norm.  
Moreover, the connected sum of two copies (with appropriate orientation) of a 
fibered link is the link associated to a real polynomial isolated singularity.

\begin{problem}
Let $M$ and $N$ be compact smooth manifolds. If $\varphi(M,N)$ is finite,  then do $\varphi_{c}(M,N)$ and the similarly defined $\varphi_{{\mathcal C}^{\omega}}(M,N)$ in the real-analytic category 
also have to be finite?  A stronger statement would be to have equalities:   
\[\varphi(M,N)=\varphi_{c}(M,N)=\varphi_{{\mathcal C}^{\omega}}(M,N)\]
\end{problem}

Note that in the small codimension case  $m\leq 2k-1$ Theorem \ref{smallcodimension} implies that 
whenever $\varphi_c(M^m, N^k)$ is finite, then indeed $\varphi(M^m,N^k)=\varphi_c(M^m,N^k)\in\{0,1\}$. 

In the relation with Lemma \ref{trivialfibration} it would be interesting to construct nontrivial 
elements in the cohomology/homotopy of the group ${\rm Homeo}(F,\partial F)$ of (PL) homeomorphisms  
of a Poenaru-Mazur manifold $F$.

{\bf Acknowledgements.} The author is grateful to Norbert A'Campo, Cornel Pintea, Nicolas Dutertre, 
Valentin Poenaru and Mihai Tib\v{a}r for useful discussions and to the referees for 
corrections and helpful remarks. 
The author was partially supported by the GDRI Eco-Math.

\bibliographystyle{amsplain}

\end{document}